\documentclass{amsart}
\usepackage{amsfonts}
\usepackage{amssymb}
\usepackage{amsmath}
\usepackage{amsthm}
\usepackage{cite}
\usepackage{wasysym}
\usepackage[all]{xy} \newdir{ >}{{}*!/-10pt/@{>}}
\usepackage{mathrsfs}
\usepackage{dsfont}
\usepackage{enumerate}
\usepackage{tikz}
\usepackage[top=3cm,bottom=3cm,left=3cm,right=3cm]{geometry}

\title[The Moduli Problem of Lobb and Zentner for Colored sl(N)]{The Moduli Problem of Lobb and Zentner\\ and the Coloured sl(N) Graph Invariant}

\author{Jonathan Grant}
\address{Dept of Mathematical Sciences\\ Durham University\\ Science Laboratories\\ South Rd.\\ Durham\\ DH1 3LE\\ UK}
\email{jonathan.grant@durham.ac.uk}

\newcommand{\C}{\mathds C}
\newcommand{\G}{\mathds G}
\renewcommand{\P}{\mathds P}

\newcommand{\M}{\mathscr{M}}

\newcommand{\chuse}[2]{\left[\begin{array}{c} #1 \\ #2\end{array}\right]}
\newtheorem{thm}{Theorem}[section]
\newtheorem{lem}{Lemma}[section]
\newtheorem{prop}{Proposition}[section]
\theoremstyle{definition}\newtheorem{remark}[thm]{Remark}
\newtheorem*{acknowledgement}{Acknowledgement}

\makeatletter
\let\c@lem=\c@thm
\makeatother
\makeatletter
\let\c@prop=\c@thm
\makeatother
\begin{document}

\begin{abstract}
Motivated by a possible connection between the $\mathrm{SU}(N)$ instanton knot Floer homology of Kronheimer and Mrowka and $\mathfrak{sl}(N)$ Khovanov-Rozansky homology, Lobb and Zentner recently introduced a moduli problem associated to colourings of trivalent graphs of the kind considered by Murakami, Ohtsuki and Yamada in their state-sum interpretation of the quantum $\mathfrak{sl}(N)$ knot polynomial. For graphs with two colours, they showed this moduli space can be thought of as a representation variety, and that its Euler characteristic is equal to the $\mathfrak{sl}(N)$ polynomial of the graph evaluated at $1$. We extend their results to graphs with arbitrary colourings by irreducible anti-symmetric representations of $\mathfrak{sl}(N)$.
\end{abstract}

\maketitle

\section{Introduction}
In their paper \cite{LZ}, Lobb and Zentner introduced a moduli space of assignments of lines and planes to oriented trivalent graphs, corresponding to the Murakami, Ohtsuki and Yamada (MOY) state model interpretation of the Reshetikhin-Turaev  quantum $\mathfrak{sl}(N)$ link polynomial. Trivalent graphs considered by MOY have a decoration on edges, indicated by numbers $i$, corresponding to anti-symmetric tensor powers $\Lambda^i V$ of the standard $N$-dimensional representation $V$ of $\mathfrak{sl}(N)$, with the condition that at a trivalent vertex the signed sum (with signs given by orientations) of decorations of each edge is $0$. Graphs considered by Lobb and Zentner are decorated with elements in $\G(i,N)$, with the condition that at each trivalent vertex the span of the decorations of the `inward' pointing edges equals the span of the decorations of the `outward' pointing edges, and decorations of edges pointing the same way are orthogonal.

In \cite{LZ}, the authors proved that if a graph $\Gamma$ is coloured with $1$'s and $2$'s, then the Euler characteristic of their moduli space is equal to the value of the MOY polynomial evaluated at $1$. In this paper, we extend that result to all higher colourings of graphs.

The plan of this paper is to give the background for the $\mathfrak{sl}(N)$ graph polynomial in Section \ref{polynomialgraphs}, then state our result in Section \ref{modulispace} and relate the moduli space to a representation variety in Section \ref{relationship}, and then prove the main theorem in Sections \ref{moves} and \ref{proofs}.

\begin{acknowledgement} I would like to thank my advisor Andrew Lobb for suggesting the topic of this paper.\end{acknowledgement}

\section{The sl(N) Polynomial for Trivalent Graphs}\label{polynomialgraphs}
In their paper \cite{MOY}, MOY introduced a polynomial associated to certain coloured trivalent oriented graphs, in order to provide a state model interpretation of the quantum $\mathfrak{sl}(N)$ polynomial for links. The trivalent graphs considered are all oriented and planar, and are considered to have a colouring in $\{1,\ldots, N-1\}$ with the signed sum (signs given by the orientation) of colourings around a trivalent vertex equal to $0$. When drawing such graphs, we will usually suppress orientations, but use curved lines to indicate the relative orientations of edges around a vertex. The $\mathfrak{sl}(N)$ polynomial for links is related to the MOY polynomial by the relations in figure \ref{knot}
\begin{figure}[h]\[
\left(\begin{tikzpicture}[baseline=-0.65ex]
\draw[->] (-1,-1) -- (1,1);
\draw (1,-1) -- (0.5, -0.5);
\draw[<-] (-1,1) -- (-0.5,0.5);
\end{tikzpicture}\right)_N = q\left(\begin{tikzpicture}[baseline=-0.65ex]
\draw[->] (-1,-1) -- (-1,1);
\draw[->] (1,-1) -- (1,1);
\draw (-1.25,0.5) node {$1$};
\draw (1.25,0.5) node {$1$};
\end{tikzpicture}\right)_N
- \left( \begin{tikzpicture}[baseline=-0.65ex]
\draw[->] (-1,-1) to [out=90,in=270] (0,-0.33) to [out=90,in=270] (0,0.33) to [out=90,in=270] (-1,1);
\draw (1,-1) to [out=90,in=270] (0,-0.33);
\draw[->] (0,0.33) to [out=90,in=270] (1,1);
\draw (-1.25, -0.8) node {$1$};
\draw (-1.25,0.8) node {$1$};
\draw (1.25, -0.8) node {$1$};
\draw (1.25, 0.8) node {$1$};
\draw (0.25,0) node {$2$};
\end{tikzpicture}\right)_N
\]

\[
\left(\begin{tikzpicture}[baseline=-0.65ex]
\draw (-1,-1) -- (-0.5,-0.5);
\draw[->] (0.5,0.5) -- (1,1);
\draw[<-] (-1,1) -- (1,-1);
\end{tikzpicture}\right)_N = q^{-1}\left(\begin{tikzpicture}[baseline=-0.65ex]
\draw[->] (-1,-1) -- (-1,1);
\draw[->] (1,-1) -- (1,1);
\draw (-1.25,0.5) node {$1$};
\draw (1.25,0.5) node {$1$};
\end{tikzpicture}\right)_N
- \left( \begin{tikzpicture}[baseline=-0.65ex]
\draw[->] (-1,-1) to [out=90,in=270] (0,-0.33) to [out=90,in=270] (0,0.33) to [out=90,in=270] (-1,1);
\draw (1,-1) to [out=90,in=270] (0,-0.33);
\draw[->] (0,0.33) to [out=90,in=270] (1,1);
\draw (-1.25, -0.8) node {$1$};
\draw (-1.25,0.8) node {$1$};
\draw (1.25, -0.8) node {$1$};
\draw (1.25, 0.8) node {$1$};
\draw (0.25,0) node {$2$};
\end{tikzpicture}\right)_N
\]\caption{MOY resolutions of knot diagrams}\label{knot}\end{figure}

The edges in the knot diagram may be understood to be coloured $1$, with the standard $N$-dimensional representation $V$ of $\mathfrak{sl}(N)$. MOY also introduced an invariant of framed coloured links, given by figures \ref{colour1} and \ref{colour2}.
\begin{figure}\[
\left(\begin{tikzpicture}[baseline=-0.65ex]
\draw[->] (-1,-1) -- (1,1);
\draw (1,-1) -- (0.5, -0.5);
\draw[<-] (-1,1) -- (-0.5,0.5);
\draw (-1,0.75) node {$i$};
\draw (1,0.75) node {$j$}; 
\end{tikzpicture}\right)_N=
\sum_{k=0}^i (-1)^{k+(j+1)i}q^{i-k}\left( 
\begin{tikzpicture}
[baseline=-0.65ex]
\draw[->] (0,-1)-- (0,1);
\draw[->] (2,-1)-- (2,1);
\draw (0,-0.33) to [out=270,in=180] (0.5,-0.5) to [out=0,in=180] (1.5,-0.5) to [out=0, in=90] (2,-0.66);
\draw (0,0.33) to [out=90, in=180] (0.5,0.5) to [out=0,in=180] (1.5,0.5) to [out=0, in=270] (2,0.66);
\draw (0,-1.25) node {$j$};
\draw (2,-1.25) node {$i$};
\draw (-0.5,0) node {$j+k$};
\draw (0,1.25) node {$i$};
\draw (1, -0.66) node {$k$};
\draw (2.5,0) node {$i-k$};
\draw (1, 0.66) node {$j+k-i$};
\draw (2,1.25) node {$j$};
\end{tikzpicture}
\right)_N
\]\caption{MOY resolutions of a coloured knot diagram if $i\leq j$}\label{colour1}\end{figure}
\begin{figure}\[
\left(\begin{tikzpicture}[baseline=-0.65ex]
\draw[->] (-1,-1) -- (1,1);
\draw (1,-1) -- (0.5, -0.5);
\draw[<-] (-1,1) -- (-0.5,0.5);
\draw (-1,0.75) node {$i$};
\draw (1,0.75) node {$j$}; 
\end{tikzpicture}\right)_N=
\sum_{k=0}^i (-1)^{k+(i+1)j}q^{j-k}\left( 
\begin{tikzpicture}
[baseline=-0.65ex]
\draw[->] (0,-1)-- (0,1);
\draw[->] (2,-1)-- (2,1);
\draw (0,-0.66) to [out=90,in=180] (0.5,-0.5) to [out=0,in=180] (1.5,-0.5) to [out=0, in=270] (2,-0.33);
\draw (0,0.66) to [out=270, in=180] (0.5,0.5) to [out=0,in=180] (1.5,0.5) to [out=0, in=90] (2,0.33);
\draw (0,-1.25) node {$j$};
\draw (2,-1.25) node {$i$};
\draw (-0.5,0) node {$j-k$};
\draw (0,1.25) node {$i$};
\draw (1, -0.66) node {$k$};
\draw (2.5,0) node {$i+k$};
\draw (1, 0.66) node {$i+k-j$};
\draw (2,1.25) node {$j$};
\end{tikzpicture}
\right)_N
\]\caption{MOY resolutions of a coloured knot diagram if $i>j$}\label{colour2}\end{figure}
Negative crossings are similar, with $q$ replaced by $q^{-1}$. 

\section{A Moduli Space of Colourings}\label{modulispace}
In their paper \cite{LZ} Lobb and Zentner introduced a moduli space $\mathscr{M}(\Gamma)$ of colourings of a diagram $\Gamma$ by associating to an $i$-coloured edge an element of the complex Grassmannian $\G(i,N)$ in such a way that if the three edges around a vertex are coloured $i$, $j$ and $i+j$, then the $i$-plane and the $j$-plane are orthogonal and span the $(i+j)$-plane in $\C^N$. They showed that if $\Gamma$ is coloured with $1$'s and $2$'s, then
\[ \chi (\mathscr{M}(\Gamma))=(\Gamma)_N(1) \]
ie. the Euler characteristic is the MOY polynomial evaluated at $1$. It is tempting to think that in fact the Poincar\'{e} polynomial of $\mathscr{M}(\Gamma)$ is equal to $(\Gamma)_N$, but Lobb and Zentner showed that this is false in general. 

In this paper, we show that the same relation holds for all higher colourings as well. \begin{thm}\label{main}
For a coloured planar trivalent graph $\Gamma$, we have
\[ \chi(\M(\Gamma))=(\Gamma)_N(1). \]
\end{thm}

\section{Relationship with Representation Varieties}\label{relationship}
Motivated by a conjectural relationship between Khovanov-Rozansky homology \cite{khovanov2008matrix} and the instanton knot Floer homology associated to $\mathrm{SU}(N)$ of Kronheimer and Mrowka \cite{kronheimer2011knot}, we can relate this moduli space to a space of representations of the fundamental group of the graph complement in $S^3$ into $\mathrm{SU}(N)$. Let $\Gamma$ be the graph in question. In analogy to the Wirtinger presentation of the knot group, the group $G_\Gamma :=\pi_1(S^3\backslash \Gamma)$ can be presented as
\[ 
\langle x_1,\ldots x_m|R_1,\ldots,R_c \rangle
\]
where $x_j$ represents a positively-oriented meridian to the $j$th edge, and the relations are given as follows: at the $i$th trivalent vertex, either two edges flow into the vertex or two edges flow out. Let $x$ and $y$ be the oriented meridians corresponding to these edges, and let $z$ be the oriented meridian corresponding to the remaining edge. Suppose, in the planar diagram, travelling in the anti-clockwise direction around the vertex meets $(x,y,z)$ in that order. Then if the two edges flow into the vertex, let $R_i=xyz^{-1}$, and if the two edges flow out let $R_i=yxz^{-1}$ (see figures \ref{flowin} and \ref{flowout}).

\begin{figure}[h]\center\begin{tikzpicture}[baseline=-0.65ex]
\draw (-1,-1) to [out=90,in=270] (0,0);
\draw[->](0,0)-- (0,1);
\draw (1,-1) to [out=90,in=270] (0,0);
\draw (-1.5,-0.9) node {$x$};
\draw (1.5,-0.9) node {$y$};
\draw (0.5,0.5) node {$z$};
\end{tikzpicture} \quad \quad $R_i=xyz^{-1}$
\caption{}\label{flowin}\end{figure}

\begin{figure}[h]
\center \begin{tikzpicture}[baseline=-0.65ex]
\draw (0,-1) -- (0,0);
\draw[->] (0,0) to [out=90,in=270] (-1,1);
\draw[->] (0,0) to [out=90,in=270] (1,1);
\draw (1.5,0.9) node {$x$};
\draw (-1.5,0.9) node {$y$};
\draw (0.5,-0.5) node {$z$};
\end{tikzpicture} \quad \quad $R_i=yxz^{-1}$
\caption{}\label{flowout}\end{figure}

We define $\zeta=\exp(i\pi/N)$, and for each $j$ we set
\[ \Phi_j=\zeta^j\mathrm{diag}(-1,-1,\cdots,-1,1,1,\cdots,1) \]
which is a diagonal $N\times N$ matrix with the first $j$ elements on the diagonal equal to $-\zeta^j$, and the last $N-j$ elements equal to $\zeta^j$.

Now we let $R_{\Phi_j}(G_\Gamma;\mathrm{SU}(N))$ be the subspace of homomorphisms $\rho:G_\Gamma\to \mathrm{SU}(N)$, with the compact-open topology, with the condition that an oriented meridian $m$ to an edge coloured by $j$ must satisfy
\[ \rho(m)\sim \Phi_j \]
ie. $\rho(m)$ is conjugate to $\Phi_j$ in $\mathrm{SU}(N)$.
\begin{lem}\label{orthog}
If $S,T\in \mathrm{SU}(N)$ are conjugate to $\Phi_i$ and $\Phi_j$ respectively, then $ST$ is conjugate to $\Phi_{i+j}$ if and only if the $(-\zeta^i)$-eigenspace of $S$ is orthogonal to the $(-\zeta^j)$-eigenspace of $T$.
\end{lem}
\begin{proof}
If $v$ is in the $(-\zeta^i)$-eigenspace of $S$, then $v$ is orthogonal to the $(-\zeta^j)$-eigenspace of $T$ and so must be contained in the $\zeta^j$-eigenspace of $T$. Hence $ST(v)=S(\zeta^jv)=-\zeta^{i+j}v$. Similarly, if $v$ is in the $(-\zeta^j)$-eigenspace of $T$, then $ST(v)=-\zeta^{i+j}v$. It follows by calculation that $ST$ has an $(i+j)$-dimensional $(-\zeta^{i+j})$-eigenspace, and a $(N-i-j)$-dimensional $\zeta^{i+j}$-eigenspace, and so is conjugate to $\Phi_{i+j}$ as required. The converse is clear from the above argument.
\end{proof}
We can therefore define a natural map
\[ D:R_{\Phi_j}(G_\Gamma;\mathrm{SU}(N))\to \M(\Gamma) \]
by assigning to an $i$-coloured edge in $\Gamma$ the $(-\zeta^i)$-eigenspace of the image of its oriented meridian in $\mathrm{SU}(N)$. At a trivalent vertex we have the relation $xy=z$, so $\rho(x)\rho(y)=\rho(z)$. As $\rho(z)\sim \Phi_{i+j}$, $\rho(x)\sim \Phi_i$ and $\rho(y)\sim \Phi_j$, Lemma \ref{orthog} implies that the $(-\zeta^i)$-eigenspace of $\rho(x)$ is orthogonal to the $(-\zeta^j)$-eigenspace of $\rho(y)$, so the assignment of these eigenspaces to the edges gives a permissible colouring of $\Gamma$.
\begin{thm}
The map \[ D:R_{\Phi_j}(G_\Gamma;\mathrm{SU}(N))\to \M(\Gamma) \] is a homeomorphism.
\end{thm}
\begin{proof}
We define an inverse as follows. Let $A_e$ be the colouring of edge $e$ coloured $i$, and let $S_{A_e}$ be the unique element in $\mathrm{SU}(N)$ that is conjugate to $\Phi_i$ and has $A_e$ as its $(-\zeta^i)$-eigenspace. Then define $\rho(m_e)=S_{A_e}$ where $m_e$ is the meridian to $e$. Because of the conditions on the colourings in $\M(\Gamma)$, this satisfies the required relations to be an element in $R_{\Phi_j}(G_\Gamma;\mathrm{SU}(N))$. It is then easy to see that this is a continuous inverse of $D$.
\end{proof}

\section{MOY moves}\label{moves}
For convenience, we fix the following notation:
\[ [k]=\frac{q^k-q^{-k}}{q-q^{-1}}=q^{k-1}+q^{k-3}+\cdots+q^{3-k}+q^{1-k}\]
\[ [k]!=[k][k-1]\cdots [2] \]
\[ \chuse{i}{j}=\frac{[i]!}{[j]![i-j]!}. \]\pagebreak
The MOY graph invariant satisfies the local relations in figure \ref{MOYmoves} (cf. MOY \cite{MOY}).

\begin{figure}[h!]\begin{equation}\label{zero}\tag{Move 0}
\left(\begin{tikzpicture}[baseline=-0.65ex]
\draw (0,0) circle (0.5);
\draw (0.7,0) node {$i$};
\end{tikzpicture}\right)_N = \chuse{N}{i}
\end{equation}
\begin{equation}\label{one}\tag{Move 1}
\left(\begin{tikzpicture}[baseline=-0.65ex,scale=0.6]
\draw (0,-1.5) -- (0,1.5);
\draw (0,-0.5) to [out=270, in=270] (1,-0.2) to [out=90, in=270] (1,0.2) to [out=90, in=90] (0,0.5);
\draw (-0.25, -1) node {$i$};
\draw (-0.25, 1) node {$i$};
\draw (-0.7, 0) node {$j+i$};
\draw (1.25,0) node {$j$};
\end{tikzpicture}\right)_N = \chuse{N-i}{j} 
\left(\begin{tikzpicture}[baseline=-0.65ex,scale=0.6]
\draw (0,-1.5) -- (0,1.5);
\draw (-0.25,0) node {$i$};
\draw (0.25,0) node {};
\end{tikzpicture}\right)_N
\end{equation}
\begin{equation}\label{two}\tag{Move 2}
\left(\begin{tikzpicture}[baseline=-0.65ex,scale=0.6]
\draw (0.5,-1.5) -- (0.5,-0.6);
\draw (0.5, -0.6) to [out=90, in=270] (0,0) to [out=90,in=270] (0.5,0.6);
\draw (0.5, -0.6) to [out=90, in=270] (1,0) to [out=90,in=270] (0.5,0.6);
\draw (0.5,0.6) -- (0.5,1.5);
\draw (0.75, -1) node {$i$};
\draw (0.75, 1.1) node {$i$};
\draw (-0.7, 0) node {$i-j$};
\draw (1.25, 0) node {$j$};
\end{tikzpicture}\right)_N = \chuse{i}{j}
\left(\begin{tikzpicture}[baseline=-0.65ex,scale=0.6]
\draw (0.5,-1.5) -- (0.5, 1.5);
\draw (0.75, 0) node {$i$};
\draw (0.25, 0) node {};
\end{tikzpicture}\right)_N
\end{equation}
\begin{equation}\label{three}\tag{Move 3}
\left(\begin{tikzpicture}[baseline=-0.65ex,scale=0.6]
\draw (0,-1.5) -- (0,-0.5);
\draw (0,-0.5) to [out=90,in=270] (-1,0.5) to [out=90,in=270] (-2,1.5);
\draw (-1,0.5) to [out=90, in=270] (0,1.5);
\draw (0,-0.5) to [out=90,in=270] (1,0.5) to [out=90,in=270] (1,1.5);
\draw (1.3,-1) node {$i+j+k$};
\draw (1.25, 1.25) node {$k$};
\draw (-1.5,0) node {$i+j$};
\draw (-2.25, 1.25) node {$i$};
\draw (-0.5,1.25) node {$j$};
\end{tikzpicture}\right)_N=
\left(\begin{tikzpicture}[baseline=-0.65ex,scale=0.6]
\draw (0,-1.5) -- (0,-0.5);
\draw (0,-0.5) to [out=90,in=270] (-1,0.5) to [out=90,in=270] (-1,1.5);
\draw (1,0.5) to [out=90, in=270] (2,1.5);
\draw (0,-0.5) to [out=90,in=270] (1,0.5) to [out=90,in=270] (0,1.5);
\draw (1.3,-1.25) node {$i+j+k$};
\draw (2.25, 1.25) node {$k$};
\draw (1.5,0) node {$j+k$};
\draw (-1.25, 1.25) node {$i$};
\draw (0.5,1.25) node {$j$};
\end{tikzpicture}\right)_N
\end{equation}
\begin{equation}\label{four}\tag{Move 4}
\left(\begin{tikzpicture}[baseline=-0.65ex]
\draw (0,-1) to [out=90,in=180] (0.5,-0.5) to [out=0,in=180] (1,-0.5) to [out=0, in=90] (1.5,-1);
\draw (0.5,-0.5) to [out=180, in=270] (0,0) to [out=90,in=180] (0.5,0.5);
\draw (1,-0.5) to [out=0, in=270] (1.5,0) to [out=90,in=0] (1,0.5);
\draw (0,1) to [out=270,in=180] (0.5,0.5) to [out=0,in=180] (1,0.5) to [out=0, in=270] (1.5,1);
\draw (-0.25, -0.75) node {$1$};
\draw (1.75, -0.75) node {$i$};
\draw (-0.25, 0) node {$i$};
\draw (-0.25, 0.75) node {$1$};
\draw (1.75, 0.75) node {$i$};
\draw (1.75, 0) node {$1$};
\draw (0.75,-0.75) node {$i+1$};
\draw (0.75,0.75) node {$i+1$};
\end{tikzpicture}\right)_N
= [N-i-1]
\left(\begin{tikzpicture}[baseline=-0.65ex]
\draw (0,-1) to [out=90, in=180] (0.5,-0.5) to [out=0,in=180] (1,-0.5) to [out=0,in=90] (1.5,-1);
\draw (0,1) to [out=270, in=180] (0.5,0.5) to [out=0,in=180] (1,0.5) to [out=0,in=270] (1.5,1);
\draw (0.75,-0.5) to [out=180, in=270] (0.25,0) to [out=90,in=180] (0.75,0.5);
\draw (-0.25, -0.75) node {$1$};
\draw (1.75, -0.75) node {$i$};
\draw (1, 0) node {$i-1$};
\draw (-0.25, 0.75) node {$1$};
\draw (1.75, 0.75) node {$i$};
\end{tikzpicture}\right)_N
+
\left(\begin{tikzpicture}[baseline=-0.65ex]
\draw (0,-1) -- (0,1);
\draw (1,-1) -- (1,1);
\draw (-0.25,0) node {$1$};
\draw (1.25, 0) node {$i$};
\end{tikzpicture}\right)_N
\end{equation}
\begin{equation}\label{five}\tag{Move 5}
\left(\begin{tikzpicture}[baseline=-0.65ex]
\draw (0,-1)-- (0,1);
\draw (2,-1)-- (2,1);
\draw (0,-0.33) to [out=270,in=180] (0.5,-0.5) to [out=0,in=180] (1.5,-0.5) to [out=0, in=90] (2,-0.66);
\draw (0,0.33) to [out=90, in=180] (0.5,0.5) to [out=0,in=180] (1.5,0.5) to [out=0, in=270] (2,0.66);
\draw (0.1,-1.25) node {$1$};
\draw (1.3,-1.25) node {$i+j-1$};
\draw (0.45,0) node {$i+k$};
\draw (0.1,1.25) node {$i$};
\draw (1, -0.75) node {$i+k-1$};
\draw (1.6,0) node {$j-k$};
\draw (1, 0.66) node {$k$};
\draw (2,1.25) node {$j$};
\end{tikzpicture}\right)_N
= \chuse{j-1}{k-1}\left(
\begin{tikzpicture}[baseline=-0.65ex]
\draw (0,-1) to [out=90,in=270] (0.5,-0.33) to [out=90,in=270] (0.5,0.33) to [out=90,in=270] (0,1);
\draw (1,-1) to [out=90,in=270] (0.5,-0.33);
\draw (1,1) to [out=270,in=90] (0.5,0.33);
\draw (0, -1.25) node {$1$};
\draw (0.9,-1.25) node {$i+j-1$};
\draw (0,1.25) node {$i$};
\draw (1,1.25) node {$j$};
\draw (1,0) node {$i+j$};
\end{tikzpicture}\right)_N
+\chuse{j-1}{k}
\left(\begin{tikzpicture}[baseline=-0.65ex]
\draw (0,-1) -- (0,1);
\draw (1,-1) -- (1,1);
\draw (0,0.33) to [out=270,in=180] (0.5,0) to [out=0,in=90] (1,-0.33);
\draw (0,1.25) node {$i$};
\draw (0,-1.25) node {$1$};
\draw (1,1.25) node {$j$};
\draw (0.9,-1.25) node {$i+j-1$};
\draw (0.5,0.3) node {$i-1$};
\end{tikzpicture}\right)_N
\end{equation}
\caption{The six MOY moves}\label{MOYmoves}\end{figure}

In all cases orientations can be chosen arbitrarily, but must be chosen consistently for two sides of a given move.

\begin{thm}\label{uniquely}
The six MOY moves uniquely determine the MOY $\mathfrak{sl}(N)$ polynomial for coloured oriented trivalent planar graphs.
\end{thm}
To prove this, we first specialise to $\{1,2\}$-coloured graphs:
\begin{prop}\label{uniquelyonetwo}
The six MOY moves, specialised to colourings in $\{1,2\}$, determine the $\mathfrak{sl}(N)$ polynomial for oriented trivalent plane graphs coloured with $\{1,2\}$.
\end{prop}
\begin{proof}
Given a diagram coloured in $\{1,2\}$, we can use \ref{zero} to remove closed loops coloured with $2$, so suppose the remaining graph $\Gamma$ has $n$ edges coloured $2$. We use the relationship with knot diagrams in Section \ref{polynomialgraphs} to construct a knot diagram $D$ with $n$ crossings for which the resolution with the most $2$-coloured edges is $\Gamma$. With appropriate choice of crossings, we can ensure that $D$ is a diagram for an unlink $U$. By the work of MOY, if the polynomial satisfies the MOY moves it also satisfies the Reidemeister moves, so the polynomial is a link invariant. Then $(U)_N=(-1)^n(\Gamma)_N+\sum_{i=1}^{2^n-1}(-1)^{k_i}(q)^{\epsilon_i}(\Gamma_i)_N$ where each $\Gamma_i$ has $k_i<n$ edges coloured with $2$, and $\epsilon_i$ is the number of positive crossings resolved into edges coloured $1$ minus the number of negative crossings resolved into edges coloured with $1$. By induction, we can use the MOY moves to calculate the values of $(\Gamma_i)_N$ for each $i$, and the $\mathfrak{sl}(N)$ polynomial for $U$ is $[N]^{d}=\left(\frac{q^N-q^{-N}}{q-q^{-1}}\right)^d$, where $d$ is the number of components of $U$. Hence we can calculate $(\Gamma)_N$ using MOY moves.
\end{proof}

\begin{proof}[Proof of Theorem \ref{uniquely}]
Suppose the largest colouring in the diagram is $m$, $m>2$. The idea is that, following Wu \cite{wu2009colored}, we replace all the edges coloured with $m$ with edges with colourings smaller than $m$.

If the diagram contains any disjoint circles, then remove them with \ref{zero}. If there are any remaining $m$ edges, then locally the diagram is
\[ \Gamma=\begin{tikzpicture}[baseline=-0.65ex,scale=0.5]
\draw (-1,-1) to [out=90,in=270] (0,-0.33) to [out=90,in=270] (0,0.33) to [out=90,in=270] (-1,1);
\draw (1,-1) to [out=90,in=270] (0,-0.33);
\draw (0,0.33) to [out=90,in=270] (1,1);
\draw (-1.25, -0.8) node {$j$};
\draw (-1.25,0.8) node {$l$};
\draw (2, -0.8) node {$m-j$};
\draw (2, 0.8) node {$m-l$};
\draw (0.5,0) node {$m$};
\end{tikzpicture}
\]
with $j,l<m$. Then we have the equality in figure \ref{step1}, using \ref{two} and \ref{three} twice.
\begin{figure}\[
(\Gamma)_N=\frac{1}{[j][l]}\left(\begin{tikzpicture}[baseline=-0.65ex,scale=0.9]
\draw (0,-3) to [out=90,in=270] (0,-2.5) to [out=135,in=270] (-0.5,-2) to [out=90,in=225] (0,-1.5) to [out=90,in=270] (1,-0.5);
\draw (0,-2.5) to [out=45,in=270] (0.5,-2) to [out=90,in=315] (0,-1.5);
\draw (1,0.5) to [out=90,in=270] (0,1.5) to [out=45,in=270] (0.5,2) to [out=90,in=315] (0,2.5) to [out=90,in=270] (0,3);
\draw (0,1.5) to [out=135,in=270] (-0.5,2) to [out=90,in=225] (0,2.5);
\draw (1,-0.5) -- (1,0.5);
\draw (2,-3) to [out=90,in=270] (1,-0.5);
\draw (2,3) to [out=270,in=90] (1,0.5);
\draw (0, -3.25) node {$j$};
\draw (2,-3.25) node {$m-j$};
\draw (0,3.25) node {$l$};
\draw (2,3.25) node {$m-l$};
\draw (1.6,0) node {$m$};
\draw (1,-2) node {$j-1$};
\draw (-0.75,-2) node {$1$};
\draw (0.1,-1) node {$j$};
\draw (-0.75,2) node {$1$};
\draw (1,2) node {$l-1$};
\draw (0.1,1) node {$l$};
\end{tikzpicture}\right)_N
= \frac{1}{[j][l]}\left( \begin{tikzpicture}[baseline=-0.65ex,scale=0.9]
\draw (0,-3) to [out=90,in=270] (0,-2) to [out=90,in=270] (1,-0.5) to [out=90,in=270] (1,0.5) to [out=90,in=270] (0,2) to [out=90,in=270] (0,3);
\draw (2,-3) to [out=90,in=270](2,-2) to [out=90,in=270]  (1,-0.5);
\draw (2,3) to [out=270,in=90] (2,2) to [out=270,in=90] (1,0.5);
\draw (0,2.7) to [out=270, in=180] (0.5,2.35) to [out=0,in=180] (1.5,2.35) to [out=0,in=90] (2,2);
\draw (0,-2.7) to [out=90,in=180] (0.5,-2.35) to [out=0,in=180] (1.5,-2.35) to [out=0,in=270] (2,-2);
\draw (0, -3.25) node {$j$};
\draw (2,-3.25) node {$m-j$};
\draw (0,3.25) node {$l$};
\draw (2,3.25) node {$m-l$};
\draw (1.6,0) node {$m$};
\draw (1,2.6) node {$l-1$};
\draw (1,-2.6) node {$j-1$};
\draw (-0.25,2) node {$1$};
\draw (-0.25,-2) node {$1$};
\draw (2,1.1) node {$m-1$};
\draw (2,-1.1) node {$m-1$};
\end{tikzpicture}\right)_N
\]\caption{}\label{step1}\end{figure}
 But then we have the relation in figure \ref{fig}
\begin{figure}
\begin{eqnarray*}\left( \begin{tikzpicture}[baseline=-0.65ex,scale=0.9]
\draw (0,-3) to [out=90,in=270] (0,-2) to [out=90,in=270] (1,-0.5) to [out=90,in=270] (1,0.5) to [out=90,in=270] (0,2) to [out=90,in=270] (0,3);
\draw (2,-3) to [out=90,in=270](2,-2) to [out=90,in=270]  (1,-0.5);
\draw (2,3) to [out=270,in=90] (2,2) to [out=270,in=90] (1,0.5);
\draw (0,2.7) to [out=270, in=180] (0.5,2.35) to [out=0,in=180] (1.5,2.35) to [out=0,in=90] (2,2);
\draw (0,-2.7) to [out=90,in=180] (0.5,-2.35) to [out=0,in=180] (1.5,-2.35) to [out=0,in=270] (2,-2);
\draw (0, -3.25) node {$j$};
\draw (1.7,-3.25) node {$m-j$};
\draw (0,3.25) node {$l$};
\draw (1.7,3.25) node {$m-l$};
\draw (1.6,0) node {$m$};
\draw (1,2.6) node {$l-1$};
\draw (1,-2.6) node {$j-1$};
\draw (0.4,1.1) node {$1$};
\draw (0.4,-1.1) node {$1$};
\draw (2,1.1) node {$m-1$};
\draw (2,-1.1) node {$m-1$};
\end{tikzpicture}\right)_N =
\left(\begin{tikzpicture}[baseline=-0.65ex,scale=0.9]
\draw (0,-3) -- (0,3);
\draw (2,-3) -- (2,3);
\draw (0,-2.7) to [out=90,in=180] (0.5,-2.35) to [out=0,in=180] (1.5,-2.35) to [out=0,in=270] (2,-2);
\draw (0,2.7) to [out=270,in=180] (0.5,2.35) to [out=0,in=180] (1.5,2.35) to [out=0,in=90] (2,2);
\draw (0,-0.3) to [out=270,in=180] (0.5,-0.75) to [out=0,in=180] (1.5,-0.75) to [out=0,in=90] (2,-1);
\draw (0,0.3) to [out=90,in=180] (0.5,0.75) to [out=0,in=180] (1.5,0.75) to [out=0,in=270] (2,1);
\draw (0, -3.25) node {$j$};
\draw (1.7,-3.25) node {$m-j$};
\draw (0,3.25) node {$l$};
\draw (1.7,3.25) node {$m-l$};
\draw (1,2.6) node {$l-1$};
\draw (1,-2.6) node {$j-1$};
\draw (-0.25,1.5) node {$1$};
\draw (-0.25,-1.5) node {$1$};
\draw (1.45,1.5) node {$m-1$};
\draw (1.45,-1.5) node {$m-1$};
\draw (-0.25,0) node {$2$};
\draw (1.45,0) node {$m-2$};
\draw (1,1) node {$1$};
\draw (1,-1) node {$1$};
\end{tikzpicture}\right)_N 
 -[m-1]
\left(\begin{tikzpicture}
[baseline=-0.65ex,scale=0.9]
\draw (0,-3) -- (0,3);
\draw (2,-3) -- (2,3);
\draw (0,-2.7) to [out=90,in=180] (0.5,-2.35) to [out=0,in=180] (1.5,-2.35) to [out=0,in=270] (2,-2);
\draw (0,2.7) to [out=270,in=180] (0.5,2.35) to [out=0,in=180] (1.5,2.35) to [out=0,in=90] (2,2);
\draw (0, -3.25) node {$j$};
\draw (1.7,-3.25) node {$m-j$};
\draw (0,3.25) node {$l$};
\draw (1.7,3.25) node {$m-l$};
\draw (1,2.6) node {$l-1$};
\draw (1,-2.6) node {$j-1$};
\draw (0.25,0) node {$1$};
\draw (1.5,0) node {$m-1$};
\end{tikzpicture}
\right)_N
\end{eqnarray*}
\caption{}\label{fig}\end{figure}
by expanding the central portion of the first term on the right hand side of figure \ref{fig} using \ref{five}. The right hand side of figure \ref{fig} has no colourings larger than $m-1$, hence we have written $(\Gamma)_N$ in terms of diagrams containing fewer $m$-colourings.

Thus for any coloured diagram $\Gamma$, there is a diagram $\Gamma'$ that is coloured only in $\{1,2\}$ such that $(\Gamma)_N=p_\Gamma(\Gamma')_N$ for some polynomial $p_\Gamma$ determined by the MOY moves. The result then follows from Proposition \ref{uniquelyonetwo}. 
\end{proof}
\begin{remark}
Proposition \ref{uniquelyonetwo} is often used implicitly in the literature, but we have been unable to find a source that gives a proof.
\end{remark}
\begin{remark}
Note that we have not used the full strength of several of the MOY moves as written above, so it is possible to further refine the list of MOY moves that are required to determine the MOY polynomial. More precisely, \ref{one} and \ref{four} appear in Proposition \ref{uniquelyonetwo}, but do not appear in the argument of Theorem \ref{uniquely} otherwise, so are only required in the case $i=j=1$. \ref{two} is only used in the case $j=1$, and \ref{five} is only used in the case $i=k=1$. However, we include the more general cases in the list because they are useful in calculations, and are used to show that the  $\mathfrak{sl}(N)$ coloured framed link invariant is invariant under the Reidemeister moves 2 and 3 in MOY's paper \cite{MOY}.
\end{remark}

\section{Proofs}\label{proofs}
In this section, we give a proof of Theorem \ref{main}. By Theorem \ref{uniquely}, it will suffice to show that $\chi(\M(\Gamma))$ satisfies the MOY moves in Section \ref{moves}, with polynomials evaluated at $1$, thus the proof of Theorem \ref{main} breaks into proofs of Lemmas \ref{move 0}, \ref{move 1}, \ref{move 2}, \ref{move 3}, \ref{move 4} and \ref{move 5}.

It is well-known that the complex Grassmannian has only even-degree homology, for instance because Schubert varieties give a CW-decomposition with only even dimensional cells. Given this fact, we show that its Poincar\'{e} polynomial has the required form for our purposes. This lemma is well-known, but we include it as the proof will use calculations using fibre bundles and the Serre spectral sequence, which are techniques that will be used throughout the rest of this paper.
\begin{lem}\label{grassmannian}
For all $1\leq k\leq n\leq N$,
\[
\pi (\G(k,n))(q)=q^{nk-k^2}\chuse{n}{k}.
\]
where $\pi(\G(k,n))(q)$ is the Poincar\'{e} polynomial of the complex Grassmannian.
\end{lem}
\begin{proof}
This statement is true whenever $k=1$ because $\G(1,n)=\P^{n-1}$. For induction, assume that it is true for $(k-1,n-1)$. Let $\G(k-1,n-1,n)$ be the flag variety of subspaces $0\subset P \subset P' \subset \C^n$ where $\dim P=k-1$ and $\dim P'=n-1$. We write $P'=l^\perp$ for a unique line $l$ in $\C^n$.

Then we have fibre bundles
\[ \xymatrix{\G(k-1,n-1) \ar@{^{(}->}[r]& \G(k-1,n-1,n) \ar[d]^{\pi_1}\\
& \P^{n-1}} \]
and 
\[ \xymatrix{\P^{k-1} \ar@{^{(}->}[r]& \G(k-1,n-1,n) \ar[d]^{\pi_2}\\
& \G(k,n)} \]
where $\pi_1$ maps $P\subset l^\perp$ to $l$, and $\pi_2$ sends $P\subset l^\perp$ to $P\oplus l$. Both of these maps are well-defined and continuous and it is clear that $\pi_1$ is a fibre bundle. \\

For $\pi_2$, let $U_I=\{P\in \G(k,n)|P\cap \mathrm{span}(e_{i_1},\ldots, e_{i_{n-k}})= \{0\} \}$ for each $n-k$ element subset $I\subset \{1,\ldots n\}$. The set $\pi_2^{-1}(U_I)$ consists of all $(k-1)$ planes $P$ in $(n-1)$-planes $l^\perp$ in $\C^n$ with $P\oplus l\in U_I$, so each $l$ in this preimage is contained in a single $k$-plane. Hence there is a map $\phi:\pi_2^{-1}(U_I)\to U_I\times \P^{k-1}:(P\subset l^\perp )\mapsto (P\oplus l, l)$, which is a homeomorphism making the diagram
\[ \xymatrix{ \pi_2^{-1}(U_I) \ar[r]^{\phi}\ar[d]_{\pi_2} & U_I\times \P^{k-1} \ar[dl]^{\operatorname{proj}_1} \\ U_I} \]
commute, where $\operatorname{proj}_1$ is projection onto first coordinate. Therefore $\pi_2$ is a fibre bundle.

Then by the Serre spectral sequence and the fact that the bases and fibres all have only even-degree homology, we find that the Poincar\'{e} polynomials satisfy
\[ \pi(\P^{k-1})\pi(\G(k,n))=\pi(\P^{n-1})\pi(\G(k-1,n-1)) \]
which implies that
\[ q^{k-1}[k]\pi(\G(k,n))=q^{n-1}[n]q^{(n-1)(k-1)-(k-1)^2}\chuse{n-1}{k-1} \]
and hence $\pi(\G(k,n))=q^{nk-k^2}\frac{[n]}{[k]}\chuse{n-1}{k-1}=q^{nk-k^2}\chuse{n}{k}$ as required.
\end{proof}

Note in particular that $\chi(\G(k,n))=\chuse{n}{k}|_{q=-1}=\chuse{n}{k}|_{q=1}=\binom{n}{k}$.

\subsection{Move 0}
\begin{lem}\label{move 0} For all $0<i\leq N$,
\[ \pi\left(\mathscr{M}\left(\begin{tikzpicture}[baseline=-0.65ex]
\draw (0,0) circle (0.25);
\draw (0.4,0) node {$i$};
\end{tikzpicture}\right) \right)= q^{Ni-i^2}\chuse{N}{i} \]
\end{lem}
\begin{proof}
Since we colour the diagram by choosing an $i$-plane in $\C^N$, it is clear that
\[ \mathscr{M}\left(\begin{tikzpicture}[baseline=-0.65ex]
\draw (0,0) circle (0.25);
\draw (0.4,0) node {$i$};
\end{tikzpicture}\right) \cong \G(i,N)
\]
and the result follows by Lemma \ref{grassmannian}.
\end{proof}

\subsection{Move 1}
\begin{lem}\label{move 1}
For all $0<i,j\leq N$,
\[\chi\left(\mathscr{M}\left(\begin{tikzpicture}[baseline=-0.65ex,scale=0.5]
\draw (0,-1.5) -- (0,1.5);
\draw (0,-0.5) to [out=270, in=270] (1,-0.2) to [out=90, in=270] (1,0.2) to [out=90, in=90] (0,0.5);
\draw (-0.25, -1) node {$i$};
\draw (-0.25, 1) node {$i$};
\draw (-1, 0) node {$j+i$};
\draw (1.25,0) node {$j$};
\end{tikzpicture}\right)\right) = \binom{N-i}{j}\chi\left(\mathscr{M}\left(\begin{tikzpicture}[baseline=-0.65ex,scale=0.5]
\draw (0,-1.5) -- (0,1.5);
\draw (-0.25,0) node {$i$};
\draw (0.25,0) node {};
\end{tikzpicture}\right)\right)\]
\end{lem}
\begin{proof}
The space $\mathscr{M}\left(\begin{tikzpicture}[baseline=-0.65ex,scale=0.5]
\draw (0,-1.5) -- (0,1.5);
\draw (0,-0.5) to [out=270, in=270] (1,-0.2) to [out=90, in=270] (1,0.2) to [out=90, in=90] (0,0.5);
\draw (-0.25, -1) node {$i$};
\draw (-0.25, 1) node {$i$};
\draw (-1, 0) node {$j+i$};
\draw (1.25,0) node {$j$};
\end{tikzpicture}\right)$ is a $\G(j,N-i)$-bundle over $\mathscr{M}\left(\begin{tikzpicture}[baseline=-0.65ex,scale=0.5]
\draw (0,-1.5) -- (0,1.5);
\draw (-0.25,0) node {$i$};
\draw (0.25,0) node {};
\end{tikzpicture}\right)$ because any permissible colouring of the left-hand diagram has the $j$-plane in the orthogonal complement of the $i$-plane, and we can choose this $j$-plane satisfying this condition arbitrarily. Therefore we have
\[\chi\left(\mathscr{M}\left(\begin{tikzpicture}[baseline=-0.65ex,scale=0.5]
\draw (0,-1.5) -- (0,1.5);
\draw (0,-0.5) to [out=270, in=270] (1,-0.2) to [out=90, in=270] (1,0.2) to [out=90, in=90] (0,0.5);
\draw (-0.25, -1) node {$i$};
\draw (-0.25, 1) node {$i$};
\draw (-1, 0) node {$j+i$};
\draw (1.25,0) node {$j$};
\end{tikzpicture}\right)\right) = \chi(\G(j,N-i)) \cdot
\chi\left(\mathscr{M}\left(\begin{tikzpicture}[baseline=-0.65ex,scale=0.5]
\draw (0,-1.5) -- (0,1.5);
\draw (-0.25,0) node {$i$};
\draw (0.25,0) node {};
\end{tikzpicture}\right) \right)\]
by the Serre spectral sequence, which gives the result.

\end{proof}
\begin{remark}
In their paper \cite{LZ}, Lobb and Zentner conjectured that $\M(\Gamma)$ has only even-degree homology, and if this holds then in fact 
\[ \pi\left(\mathscr{M}\left(\begin{tikzpicture}[baseline=-0.65ex,scale=0.5]
\draw (0,-1.5) -- (0,1.5);
\draw (0,-0.5) to [out=270, in=270] (1,-0.2) to [out=90, in=270] (1,0.2) to [out=90, in=90] (0,0.5);
\draw (-0.25, -1) node {$i$};
\draw (-0.25, 1) node {$i$};
\draw (-1, 0) node {$j+i$};
\draw (1.25,0) node {$j$};
\end{tikzpicture}\right) \right) = \pi(\G(j,N-i)) \cdot
\pi\left(\mathscr{M}\left(\begin{tikzpicture}[baseline=-0.65ex,scale=0.5]
\draw (0,-1.5) -- (0,1.5);
\draw (-0.25,0) node {$i$};
\draw (0.25,0) node {};
\end{tikzpicture}\right) \right)=\chuse{N-i}{j}\pi\left(\mathscr{M}\left(\begin{tikzpicture}[baseline=-0.65ex,scale=0.5]
\draw (0,-1.5) -- (0,1.5);
\draw (-0.25,0) node {$i$};
\draw (0.25,0) node {};
\end{tikzpicture}\right) \right)\]
by the Serre spectral sequence. A similar statement would hold in the case of Lemma \ref{move 2} also.
\end{remark}

\subsection{Move 2}
\begin{lem}\label{move 2}
For all $0<j\leq i\leq N$,
\[ \chi\left(\M \left(\begin{tikzpicture}[baseline=-0.65ex,scale=0.6]
\draw (0.5,-1.5) -- (0.5,-0.6);
\draw (0.5, -0.6) to [out=90, in=270] (0,0) to [out=90,in=270] (0.5,0.6);
\draw (0.5, -0.6) to [out=90, in=270] (1,0) to [out=90,in=270] (0.5,0.6);
\draw (0.5,0.6) -- (0.5,1.5);
\draw (0.75, -1) node {$i$};
\draw (0.75, 1.1) node {$i$};
\draw (-0.7, 0) node {$j-i$};
\draw (1.25, 0) node {$j$};
\end{tikzpicture}\right) \right)= \binom{i}{j}\chi\left(\M
\left(\begin{tikzpicture}[baseline=-0.65ex,scale=0.6]
\draw (0.5,-1.5) -- (0.5, 1.5);
\draw (0.75, 0) node {$i$};
\draw (0.25, 0) node {};
\end{tikzpicture}\right) \right). \]
\end{lem}
\begin{proof}
Clearly $\M \left(\begin{tikzpicture}[baseline=-0.65ex,scale=0.5]
\draw (0.5,-1.5) -- (0.5,-0.6);
\draw (0.5, -0.6) to [out=90, in=270] (0,0) to [out=90,in=270] (0.5,0.6);
\draw (0.5, -0.6) to [out=90, in=270] (1,0) to [out=90,in=270] (0.5,0.6);
\draw (0.5,0.6) -- (0.5,1.5);
\draw (1, -1) node {$i$};
\draw (1, 1.1) node {$i$};
\draw (-1, 0) node {$j-i$};
\draw (1.5, 0) node {$j$};
\end{tikzpicture}\right)$ is a $\G(j,i)$-bundle over $\M
\left(\begin{tikzpicture}[baseline=-0.65ex,scale=0.5]
\draw (0.5,-1.5) -- (0.5, 1.5);
\draw (1, 0) node {$i$};
\draw (0.25, 0) node {};
\end{tikzpicture}\right) $ because the projection map sends a colouring of the left-hand side to the same colouring of the right-hand side with the bifurcation forgotten, and any choice of $j$-plane out of the $i$-plane will induce the same colouring on $\M
\left(\begin{tikzpicture}[baseline=-0.65ex,scale=0.5]
\draw (0.5,-1.5) -- (0.5, 1.5);
\draw (1, 0) node {$i$};
\draw (0.25, 0) node {};
\end{tikzpicture}\right) $.
\end{proof}

\subsection{Move 3}
\begin{lem}\label{move 3} For all $i,j,k$, we have a homeomorphism as follows:
\[ \M\left(\begin{tikzpicture}[baseline=-0.65ex,scale=0.6]
\draw (0,-1.5) -- (0,-0.5);
\draw (0,-0.5) to [out=90,in=270] (-1,0.5) to [out=90,in=270] (-2,1.5);
\draw (-1,0.5) to [out=90, in=270] (0,1.5);
\draw (0,-0.5) to [out=90,in=270] (1,0.5) to [out=90,in=270] (1,1.5);
\draw (1.2,-1) node {$i+j+k$};
\draw (1.25, 1.25) node {$k$};
\draw (-1.4,0) node {$i+j$};
\draw (-2.25, 1.25) node {$i$};
\draw (-0.5,1.25) node {$j$};
\end{tikzpicture}\right)\cong
\M\left(\begin{tikzpicture}[baseline=-0.65ex,scale=0.6]
\draw (0,-1.5) -- (0,-0.5);
\draw (0,-0.5) to [out=90,in=270] (-1,0.5) to [out=90,in=270] (-1,1.5);
\draw (1,0.5) to [out=90, in=270] (2,1.5);
\draw (0,-0.5) to [out=90,in=270] (1,0.5) to [out=90,in=270] (0,1.5);
\draw (1.2,-1.25) node {$i+j+k$};
\draw (2.25, 1.25) node {$k$};
\draw (1.5,0) node {$j+k$};
\draw (-1.25, 1.25) node {$i$};
\draw (0.5,1.25) node {$j$};
\end{tikzpicture}\right) \]
\end{lem}

\begin{proof}
A permissible colouring of the left-hand diagram is given by three mutually orthogonal planes of dimensions $i$, $j$ and $k$, which is exactly the same as the right-hand side, so the moduli spaces are equal.
\end{proof}

\subsection{Move 4}
\begin{lem}\label{move 4}For all $i\geq 1$, we have the following:
\begin{eqnarray*} \chi\left(\M\left(\begin{tikzpicture}[baseline=-0.65ex]
\draw (0,-1) to [out=90,in=180] (0.5,-0.5) to [out=0,in=180] (1,-0.5) to [out=0, in=90] (1.5,-1);
\draw (0.5,-0.5) to [out=180, in=270] (0,0) to [out=90,in=180] (0.5,0.5);
\draw (1,-0.5) to [out=0, in=270] (1.5,0) to [out=90,in=0] (1,0.5);
\draw (0,1) to [out=270,in=180] (0.5,0.5) to [out=0,in=180] (1,0.5) to [out=0, in=270] (1.5,1);
\draw (-0.25, -0.75) node {$1$};
\draw (1.75, -0.75) node {$i$};
\draw (-0.25, 0) node {$i$};
\draw (-0.25, 0.75) node {$1$};
\draw (1.75, 0.75) node {$i$};
\draw (1.75, 0) node {$1$};
\draw (0.75,-0.75) node {$i+1$};
\draw (0.75,0.75) node {$i+1$};
\end{tikzpicture}\right) \right)
&=& \chi \left(\M \left(\begin{tikzpicture}[baseline=-0.65ex]
\draw (0,-1) -- (0,1);
\draw (1,-1) -- (1,1);
\draw (-0.25,0) node {$1$};
\draw (1.25, 0) node {$i$};
\end{tikzpicture}\right) \right)
\\
&& +
(N-i-1)\chi\left(\M
\left(\begin{tikzpicture}[baseline=-0.65ex]
\draw (0,-1) to [out=90, in=180] (0.5,-0.5) to [out=0,in=180] (1,-0.5) to [out=0,in=90] (1.5,-1);
\draw (0,1) to [out=270, in=180] (0.5,0.5) to [out=0,in=180] (1,0.5) to [out=0,in=270] (1.5,1);
\draw (0.75,-0.5) to [out=180, in=270] (0.25,0) to [out=90,in=180] (0.75,0.5);
\draw (-0.25, -0.75) node {$1$};
\draw (1.75, -0.75) node {$i$};
\draw (1, 0) node {$i-1$};
\draw (-0.25, 0.75) node {$1$};
\draw (1.75, 0.75) node {$i$};
\end{tikzpicture}\right) \right)
 \end{eqnarray*}
\end{lem}

\begin{proof}
Let $\Gamma$ denote the diagram in the left-hand side, and $\Gamma_1$ and $\Gamma_2$ denote the two diagrams in the right-hand side respectively. Consider the decorations in figure \ref{fig4}
\begin{figure}[h]\[ \begin{tikzpicture}[baseline=-0.65ex,scale=1.5]
\draw (0,-1) to [out=90,in=180] (0.5,-0.5) to [out=0,in=180] (1,-0.5) to [out=0, in=90] (1.5,-1);
\draw (0.5,-0.5) to [out=180, in=270] (0,0) to [out=90,in=180] (0.5,0.5);
\draw (1,-0.5) to [out=0, in=270] (1.5,0) to [out=90,in=0] (1,0.5);
\draw (0,1) to [out=270,in=180] (0.5,0.5) to [out=0,in=180] (1,0.5) to [out=0, in=270] (1.5,1);
\draw (-0.125, -0.75) node {$\beta$};
\draw (-0.125, 0.75) node {$\alpha$};
\draw (1.625, 0.75) node {$A_i$};
\draw (1.625,0) node {$\gamma$};
\end{tikzpicture} \] \caption{}\label{fig4}\end{figure}
where $\alpha,\beta,\gamma\in \P^{N-1}$ and $A_i\in \G(i,N)$. There are three possibilities:
\begin{enumerate}[(i)]
\item $\alpha\subset A_i$ and $\alpha=\beta$
\item $\alpha\subset A_i$ and $\alpha\neq\beta$
\item $\alpha\perp A_i$.
\end{enumerate}
To colour $\Gamma$ in case (i), we can choose $\gamma$ freely as a line in the orthogonal complement to $A_i$, and once we have made this choice the rest of the colouring is determined. If we write $A_i=\alpha\oplus A_{i-1}$, then we see that there is a unique decoration of each of $\Gamma_1$ and $\Gamma_2$, with decorations given as follows:
\[ \begin{tikzpicture}[baseline=-0.65ex]
\draw (0,-1) to [out=90, in=180] (0.5,-0.5) to [out=0,in=180] (1,-0.5) to [out=0,in=90] (1.5,-1);
\draw (0,1) to [out=270, in=180] (0.5,0.5) to [out=0,in=180] (1,0.5) to [out=0,in=270] (1.5,1);
\draw (0.75,-0.5) to [out=180, in=270] (0.25,0) to [out=90,in=180] (0.75,0.5);
\draw (-0.25, -0.75) node {$\alpha$};
\draw (1.75, -0.75) node {$A_i$};
\draw (1, 0) node {$A_{i-1}$};
\draw (-0.25, 0.75) node {$\alpha$};
\draw (1.75, 0.75) node {$A_i$};
\end{tikzpicture}\quad \mathrm{and} \quad \begin{tikzpicture}[baseline=-0.65ex]
\draw (0,-1) -- (0,1);
\draw (1,-1) -- (1,1);
\draw (-0.25,0) node {$\alpha$};
\draw (1.25, 0) node {$A_i$};
\end{tikzpicture}  \]
In case (ii), we must choose $\gamma$ to be a line orthogonal to both $A_i$ and $\beta$. There is no colouring of $\Gamma_2$, and a unique colouring of $\Gamma_1$. Letting $A_i=A_{i-1}\oplus\alpha$, the colourings are as in figure \ref{fig5}.

\begin{figure}[h]\[ \begin{tikzpicture}[baseline=-0.65ex,scale=2]
\draw (0,-1) to [out=90,in=180] (0.5,-0.5) to [out=0,in=180] (1,-0.5) to [out=0, in=90] (1.5,-1);
\draw (0.5,-0.5) to [out=180, in=270] (0,0) to [out=90,in=180] (0.5,0.5);
\draw (1,-0.5) to [out=0, in=270] (1.5,0) to [out=90,in=0] (1,0.5);
\draw (0,1) to [out=270,in=180] (0.5,0.5) to [out=0,in=180] (1,0.5) to [out=0, in=270] (1.5,1);
\draw (-0.25, -0.75) node {$\beta$};
\draw (1.9, -0.75) node {$A_{i-1}\oplus\beta$};
\draw (-0.4, 0) node {$A_{i-1}\oplus\gamma$};
\draw (-0.25, 0.75) node {$\alpha$};
\draw (1.75, 0.75) node {$A_i$};
\draw (1.66, 0) node {$\gamma$};
\draw (0.75,0.66) node {$A_i\oplus\gamma$};
\draw (0.75,-0.66) node {$A_{i-1}\oplus\beta\oplus\gamma$};
\end{tikzpicture} \quad \mathrm{and} \quad
\begin{tikzpicture}[baseline=-0.65ex]
\draw (0,-1) to [out=90, in=180] (0.5,-0.5) to [out=0,in=180] (1,-0.5) to [out=0,in=90] (1.5,-1);
\draw (0,1) to [out=270, in=180] (0.5,0.5) to [out=0,in=180] (1,0.5) to [out=0,in=270] (1.5,1);
\draw (0.75,-0.5) to [out=180, in=270] (0.25,0) to [out=90,in=180] (0.75,0.5);
\draw (-0.25, -0.75) node {$\beta$};
\draw (2.25, -0.75) node {$A_{i-1}\oplus \beta$};
\draw (1, 0) node {$A_{i-1}$};
\draw (-0.25, 0.75) node {$\alpha$};
\draw (1.75, 0.75) node {$A_{i}$};
\end{tikzpicture}
\]\caption{}\label{fig5}\end{figure}

In case (iii), the colourings on $\Gamma$ and $\Gamma_2$ are uniquely determined by the data $\alpha,A_i$, with colourings as in figure \ref{fig6}.

\begin{figure}[h]\[ \begin{tikzpicture}[baseline=-0.65ex,scale=2]
\draw (0,-1) to [out=90,in=180] (0.5,-0.5) to [out=0,in=180] (1,-0.5) to [out=0, in=90] (1.5,-1);
\draw (0.5,-0.5) to [out=180, in=270] (0,0) to [out=90,in=180] (0.5,0.5);
\draw (1,-0.5) to [out=0, in=270] (1.5,0) to [out=90,in=0] (1,0.5);
\draw (0,1) to [out=270,in=180] (0.5,0.5) to [out=0,in=180] (1,0.5) to [out=0, in=270] (1.5,1);
\draw (-0.25, -0.75) node {$\alpha$};
\draw (1.66, -0.75) node {$A_{i}$};
\draw (-0.25, 0) node {$A_{i}$};
\draw (-0.25, 0.75) node {$\alpha$};
\draw (1.66, 0.75) node {$A_i$};
\draw (1.66, 0) node {$\alpha$};
\draw (0.75,0.66) node {$A_i\oplus\alpha$};
\draw (0.75,-0.66) node {$A_{i}\oplus\alpha$};
\end{tikzpicture}
\quad \mathrm{and} \quad
\begin{tikzpicture}[baseline=-0.65ex]
\draw (0,-1) -- (0,1);
\draw (1,-1) -- (1,1);
\draw (-0.25,0) node {$\alpha$};
\draw (1.25, 0) node {$A_i$};
\end{tikzpicture} \]\caption{}\label{fig6}\end{figure}

Treating $\M(\Gamma)$ as a real projective algebraic variety, evaluation at $\alpha$, $\beta$ and $A_i$ gives an algebraic map from each of $\M(\Gamma)$,$\M(\Gamma_1)$ and $\M(\Gamma_2)$ to $\P^{N-1}\times\P^{N-1}\times \G(i,N)$. Let $V$, $V_1$, $V_2$ be the respective preimages of the subvariety $\Delta$ formed by the condition that $\alpha\subset A_i$ and $\alpha=\beta$. Then by Hardt triviality there exists an open set $U\subset \P^{N-1}\times \P^{N-1}\times \G(i,N)$ containing $\Delta$ such that the respective preimages $\tilde{V}$, $\tilde{V_1}$ and $\tilde{V_2}$ of $U$ are homotopy equivalent to $V$, $V_1$ and $V_2$ respectively.

Since case (i) implies that colourings are uniquely determined for each of $\Gamma_1$ and $\Gamma_2$, we have $V_1=V_2$ naturally. Also, since the colouring of $\Gamma$ is determined by the choice of $\gamma$ in the orthogonal complement to $A_i$, $V$ is a $\P^{N-i-1}$-bundle over $V_1=V_2$, hence
\[
\chi(\tilde{V})=\chi(V)=(N-i)\chi(V_1)=(N-i-1)\chi(V_1)+\chi(V_2)=(N-i-1)\chi(\tilde{V_1})+\chi(\tilde{V_2}). \]
Now,
\begin{eqnarray*}
\chi(\M(\Gamma))&=&\chi(\M(\Gamma)\backslash V)+\chi(\tilde{V})-\chi((\M(\Gamma)\backslash V)\cap \tilde{V})\\
&=& \chi(\M(\Gamma)\backslash V)+(N-i-1)\chi(\tilde{V}_1)+\chi(\tilde{V}_2)-\chi(\tilde{V}\backslash V).
\end{eqnarray*}
For case (ii), there is a colouring of $\Gamma_1$ but no colouring of $\Gamma_2$, and for case (iii) there is a colouring for $\Gamma_2$ but no colouring of $\Gamma_1$. Hence $\M(\Gamma)\backslash V$ is the disjoint union of $\M(\Gamma_2)\backslash V_2$ and a $\P^{N-i-2}$-bundle over $\M(\Gamma_1)\backslash V_1$, because of the number of colourings of $\Gamma$ corresponding to cases (ii) and (iii). Similarly, $\tilde{V}\backslash V$ is a disjoint union of $\tilde{V_2}\backslash V_2$ and a $\P^{N-i-2}$-bundle over $\tilde{V_1}\backslash V_1$. Therefore,
\begin{eqnarray*}
\chi(\M(\Gamma)) &=& \chi(\M(\Gamma_2)\backslash V_2)+\chi(\tilde{V}_2)-\chi(\tilde{V}_2\backslash V_2) \\
&+& (N-i-1)(\chi(\M(\Gamma_1))+\chi(\tilde{V}_1)-\chi(\tilde{V}_1\backslash V_1))\\
&=& \chi(\M(\Gamma_2))+(N-i-1)\chi(\M(\Gamma_1))
\end{eqnarray*}
as required.
\end{proof}

\subsection{Move 5}
\begin{lem}\label{move 5} For $j>k\geq 1$, we have the equation in figure \ref{fig8}.
\begin{figure}[h]\begin{eqnarray*}\chi\left(\M\left(\begin{tikzpicture}[baseline=-0.65ex]
\draw (0,-1)-- (0,1);
\draw (2,-1)-- (2,1);
\draw (0,-0.33) to [out=270,in=180] (0.5,-0.5) to [out=0,in=180] (1.5,-0.5) to [out=0, in=90] (2,-0.66);
\draw (0,0.33) to [out=90, in=180] (0.5,0.5) to [out=0,in=180] (1.5,0.5) to [out=0, in=270] (2,0.66);
\draw (0,-1.25) node {$1$};
\draw (1.5,-1.25) node {$i+j-1$};
\draw (0.45,0) node {$i+k$};
\draw (0,1.25) node {$i$};
\draw (1, -0.75) node {$i+k-1$};
\draw (1.6,0) node {$j-k$};
\draw (1, 0.66) node {$k$};
\draw (2,1.25) node {$j$};
\end{tikzpicture}\right) \right)
&=& \binom{j-1}{k-1} \chi\left(\M\left(
\begin{tikzpicture}[baseline=-0.65ex]
\draw (0,-1) to [out=90,in=270] (0.5,-0.33) to [out=90,in=270] (0.5,0.33) to [out=90,in=270] (0,1);
\draw (1,-1) to [out=90,in=270] (0.5,-0.33);
\draw (1,1) to [out=270,in=90] (0.5,0.33);
\draw (0, -1.25) node {$1$};
\draw (1,-1.25) node {$i+j-1$};
\draw (0,1.25) node {$i$};
\draw (1,1.25) node {$j$};
\draw (1,0) node {$i+j$};
\end{tikzpicture}\right) \right)\\
&&
+\binom{j-1}{k}
\chi\left(\M\left(\begin{tikzpicture}[baseline=-0.65ex]
\draw (0,-1) -- (0,1);
\draw (1,-1) -- (1,1);
\draw (0,0.33) to [out=270,in=180] (0.5,0) to [out=0,in=90] (1,-0.33);
\draw (0,1.25) node {$i$};
\draw (0,-1.25) node {$1$};
\draw (1,1.25) node {$j$};
\draw (0.8,-1.25) node {$i+j-1$};
\draw (0.5,0.3) node {$i-1$};
\end{tikzpicture}\right) \right) \end{eqnarray*}\caption{}\label{fig8}\end{figure}
\end{lem}
\begin{proof}
As before, let $\Gamma$ denote the diagram on the left-hand side, and $\Gamma_1$ and $\Gamma_2$ denote the two diagrams on the right-hand side respectively. Consider the colouring of $\Gamma$ shown in figure \ref{fig7}.

\begin{figure}[h]\[
\begin{tikzpicture}[baseline=-0.65ex,scale=1.5]
\draw (0,-1)-- (0,1);
\draw (2,-1)-- (2,1);
\draw (0,-0.33) to [out=270,in=180] (0.5,-0.5) to [out=0,in=180] (1.5,-0.5) to [out=0, in=90] (2,-0.66);
\draw (0,0.33) to [out=90, in=180] (0.5,0.5) to [out=0,in=180] (1.5,0.5) to [out=0, in=270] (2,0.66);
\draw (-0.4,-0.75) node {$\alpha$};
\draw (2.5,-0.75) node {$A_{i+j-1}$};

\draw (-0.4,0.75) node {$A_i$};

\draw (1, 0.66) node {$A_k$};
\draw (2.5,0.75) node {$A_j$};
\end{tikzpicture}\]\caption{}\label{fig7}\end{figure}

Again, there are three cases to consider:
\begin{enumerate}[(i)]
\item $\alpha\subset A_i$ and $\alpha\perp A_{i+j-1}$
\item $\alpha\subset A_j$ and $\alpha\perp A_i$
\item $\alpha\subset A_i$ and $\alpha$ not orthogonal to $A_{i+j-1}$
\end{enumerate}
To give a colouring of $\Gamma$ in case (i), a fixed $(i-1)$-plane equal to $A_i\backslash \alpha$ must flow left into $A_i$, so the $k$-plane $A_k$ can be chosen arbitrarily from the orthogonal complement of $A_i\backslash \alpha$ in $A_{i+j-1}$, which is a $j$-plane. In this case, there is a unique colouring of each of $\Gamma_1$ and $\Gamma_2$.

In case (ii), we must have $\alpha\subset A_k$, so a colouring of $\Gamma$ is given by choosing a $(k-1)$-plane from the orthogonal complement of $A_i$ in $A_{i+j-1}$, which is a $(j-1)$-plane. In this case, there is a unique colouring of $\Gamma_1$, but no colouring of $\Gamma_2$.

In case (iii), $\alpha\perp A_k$ so a colouring of $\Gamma$ is given by choosing a $k$-plane from the orthogonal complement of $A_i\backslash
\alpha$ in $\alpha^\perp\cap A_{i+j-1}$, which is a $(j-1)$-plane. In this case, there is no permissible colouring of $\Gamma_1$, but a unique colouring of $\Gamma_2$.

As before, evaluation at $\alpha$, $A_i$, $A_j$ and $A_{i+j-1}$ gives an algebraic map from each of $\M(\Gamma)$, $\M(\Gamma_1)$, $\M(\Gamma_2)$ to $\P^{N-1}\times \G(i,N)\times \G(j,N)\times \G(i+j-1,N)$. Let $\Delta$ be the subvariety given by the condition that $\alpha\subset A_i$ and $\alpha\perp A_{i+j-1}$, and let $V$, $V_1$ and $V_2$ be the preimages of this set in each of $\M(\Gamma)$, $\M(\Gamma_1)$ and $\M(\Gamma_2)$. By Hardt triviality, we find an open set $U$ containing $\Delta$ such that the preimages $\tilde{V}$, $\tilde{V_1}$, $\tilde{V_2}$ are homotopic to $V$, $V_1$, $V_2$ respectively.

In case (i) there is a unique colouring of both $\Gamma_1$ and $\Gamma_2$, so $V_1=V_2$, and $V$ is a $\G(k,j)$-bundle over $V_1=V_2$. Hence
\[ \chi(V)=\binom{j}{k}\chi(V_1)=\binom{j-1}{k-1}\chi(V_1)+\binom{j-1}{k}\chi(V_2) \]
using binomial identities.

As before, outside of $V$ we have colourings corresponding to cases (ii) and (iii), which induce colourings on exactly one of $\Gamma_1$ and $\Gamma_2$. Hence $\M(\Gamma)\backslash V$ is a disjoint union of a $\G(k,j-1)$-bundle over $\M(\Gamma_2)\backslash V_2$ and a $\G(k-1,j-1)$-bundle over $\M(\Gamma_1)\backslash V_1$, with similar relations between $\tilde{V}\backslash V$ and $\tilde{V_1}\backslash V_1$ and $\tilde{V}_2\backslash V_2$. Hence
\begin{eqnarray*}
\chi(\M(\Gamma)) &=& \chi(\M(\Gamma)\backslash V)+\chi(V)-\chi(\tilde{V}\backslash V)\\
&=& \binom{j-1}{k-1}(\chi(\M(\Gamma_1)\backslash V_1)+\chi(V_1)-\chi(\tilde{V}_1\backslash V_1))\\
&& \, +\binom{j-1}{k}(\chi(\M(\Gamma_2)\backslash V_2)+\chi(V_2)-\chi(\tilde{V}_2\backslash V_2))\\
&=& \binom{j-1}{k-1}\chi(\M(\Gamma_1))+\binom{j-1}{k}\chi(\M(\Gamma_2))
\end{eqnarray*}
as required.
\end{proof}
This concludes the proof of Theorem \ref{main}.

\bibliographystyle{plain}
\bibliography{bibliography}

\end{document}